\documentclass[12pt]{amsart}
\usepackage[utf8]{inputenc}
\usepackage{amsmath,amssymb,amsthm}
\usepackage{url}
\usepackage{hyperref}

\theoremstyle{definition}
\newtheorem{definition}{Definition}[section]

\newtheorem{question}[definition]{Question}
\newtheorem{problem}[definition]{Problem}

\newtheorem{claim}[definition]{Claim}

\theoremstyle{plain}
\newtheorem{theorem}[definition]{Theorem}
\newtheorem{proposition}[definition]{Proposition}
\newtheorem{lemma}[definition]{Lemma}
\newtheorem{corollary}[definition]{Corollary}

\newtheorem{remark}[definition]{Remark}

\title[More Ramsey theory for highly connected subgraphs]{More Ramsey theory for highly connected monochromatic subgraphs}
\author{Michael Hru\v{s}\'{a}k, Saharon Shelah and Jing Zhang}
\date{}

\address{Centro de Ciencas Matem\'aticas\\
UNAM\\
A.P. 61-3, Xangari, Morelia, Michoac\'an\\
58089, M\'exico}

\email{michael@matmor.unam.mx}

\address{Department of Mathematics, Rutgers University, Hill Center, Piscataway, New Jersey, U.S.A. 08854-8019}

\address{Institute of Mathematics, Hebrew University, Givat Ram, Jerusalem 91904, Israel}

\email{shelah@math.rutgers.edu}

\address{Department of Mathematics,
University of Toronto,
Bahen Centre,
40 St. George Street, Room 6290,
Toronto, Ontario,
Canada,
M5S 2E4}

\email{jingzhan@alumni.cmu.edu}

\thanks{
{Research of the first author was partially supported  by a PAPIIT grant IN101323 and CONACyT grant A1-S-16164.
Research of the second author was partially supported by the NSF grant DMS 1833363 and by the Israel Science Foundation (ISF) grant 1838/19.
Research of the third author was supported by the  European Research Council (grant
agreement ERC-2018-StG 802756), NSERC grants RGPIN-2019-04311, RGPIN-2021-03549 and RGPIN-2016-06541. The paper appears as number 1242 on the second author's publications list.
}}

\begin{document}

\begin{abstract} An infinite graph is said to be highly connected if the induced subgraph on the complement of any set of vertices of smaller size is connected. We continue the study of  weaker versions of Ramsey's Theorem  on uncountable cardinals asserting that 
if we color edges of the complete graph we can find a large highly connected monochromatic subgraph. In particular, several  questions of Bergfalk, Hru\v{s}\'ak and Shelah  \cite{BergfalkHrusakShelah}
are answered by showing that assuming the consistency of suitable large cardinals the following are relatively consistent with {\sf{ZFC}}:
\begin{itemize}
\item $\kappa\to_{hc} (\kappa)^2_\omega$ for every regular cardinal $\kappa\geq \aleph_2$, 
\item $\neg\mathsf{CH}+ \aleph_2 \to_{hc} (\aleph_1)^2_\omega$.
\end{itemize}
Building on a work of Lambie-Hanson \cite{lambiehanson2022note}, we also show that
\begin{itemize}
\item $\aleph_2 \to_{hc} [\aleph_2]^2_{\omega,2}$ is consistent with $\neg\mathsf{CH}$.
\end{itemize}
To prove these results, we use the existence of ideals with strong combinatorial properties after collapsing suitable large cardinals.
\end{abstract}

\maketitle

\section{Introduction}

The paper studies weak versions of the Ramsey Theorem on uncountable cardinals. Following \cite{BergfalkHrusakShelah} we say that a graph $G=(X,E)$ is \emph{highly connected} if for every $Y\subseteq X$ of cardinality strictly smaller than $|X|$, the subgraph $(X\setminus Y , E\cap [X\setminus Y]^2)$ is connected.  Given cardinal numbers $\theta \leq \lambda \leq \kappa$,
$$\kappa\to_{hc}(\lambda)^2_\theta$$
denotes the statement that for every $c:[\kappa]^2\to\theta$ there is an $\xi\in\theta$  and $A\in [\kappa]^\lambda$ such that $(A, c^{-1}(\xi)\cap [A]^2)$ is highly connected, in  which case we shall say that
\emph{$A$ is highly connected in color $\xi$}.

The original motivation for studying this partition relation came from the study of higher derived limits in forcing extensions (see \cite{BergfalkLambie, BerfalkHrusakLambie, BannisterBergfalkMooreTodorcevic} and \cite{VelickovicVignati}) and is related in spirit to both the \emph{partition hypotheses} of \cite{BannisterBergfalkMooreTodorcevic} and another weak Ramsey property concerning the so-called \emph{topological $K_\kappa$} (first studied by Erd\H{o}s and Hajnal in \cite{ErdosHajnal}) considered by Komj\'ath and Shelah in \cite{KomjathShelah}.

This  paper continues the  study initiated in \cite{BergfalkHrusakShelah} and further investigated in \cite{bergfalk, lambiehanson2022note} of those cardinals $\theta < \lambda \leq \kappa$ for which the $\to_{hc}$ arrow holds. Among the facts proved in   \cite{BergfalkHrusakShelah} are:
\begin{itemize}
\item  For every infinite cardinal  $\kappa$ and natural number $n$, $\kappa\to_{hc}(\kappa)^2_n$,
\item  If $\kappa\leq 2^\theta$ then $\kappa\not\to_{hc}(\kappa)^2_\theta$, and 
\item If $\lambda=\lambda^\theta$ then $\lambda^+\to_{hc}(\lambda)^2_\theta$;
\end{itemize}
in particular, $2^\omega\not\to_{hc} (2^\omega)^2_\omega$. On the other hand, it was shown there that 
assuming  the existence of a weakly compact cardinal, consistently $2^{\omega_1}\to_{hc}(2^{\omega_1})^2_\omega$. However, in the model witnessing this relation, $2^{\omega_1}$ is weakly inaccessible, being a former large cardinal in a forcing extension by a poset satisfying a small chain condition.

In light of the results mentioned above, the following natural questions were raised in \cite{BergfalkHrusakShelah} and \cite{bergfalk}:

\begin{question} 
${ 		}$
\begin{enumerate}
\item Is it consistent that $\kappa\to_{hc}(\kappa)^2_{\omega}$ holds for an accessible $\kappa$? For example, when $\kappa$ is $\aleph_2$ or $\aleph_{\omega+1}$?
\item Is $\aleph_2\to_{hc}(\aleph_1)^2_\omega$ equivalent to the Continuum Hypothesis?
\end{enumerate}
\end{question}

Note that by a result of Lambie-Hanson \cite{lambiehanson2022note}, if $\kappa\to_{hc}(\kappa)^2_{\omega}$ holds, then $\square(\kappa)$ necessarily fails. In particular, if $\aleph_2\to_{hc}(\aleph_2)^2_{\omega}$ were to hold, our arguments would require at least a weakly compact cardinal and if $\aleph_{\omega+1}\to_{hc}(\aleph_{\omega+1})^2_{\omega}$ were to hold, then our arguments would require significantly stronger large cardinals.

Here we answer the first question in the positive and the second question in the negative
 by analyzing \emph{remnants of large cardinal properties} on smaller cardinals after suitable forcing, often the Levy or Mitchell collapse, in the form of the existence of ideals having strong combinatorial properties.

We finish this section with a few more definitions and notations. Let $\kappa$ be a regular uncountable cardinal.

\begin{definition}
Fix $k\in \omega$. We let 
$$\kappa\to [\kappa]^2_{\omega,k}$$
to abbreviate the assertion that for any $c: [\kappa]^2\to \omega$, there exist $H\in [\kappa]^{\kappa}$ and $K \in [\omega]^k$ such that $(H, c^{-1}(K))\cap [H]^2)$ is highly connected.
\end{definition}

The following is a more refined variation of the highly connected partition relations, conditioned on the lengths of the paths.

\begin{definition}
Fix $n\in \omega$. Let
$\kappa\to_{hc, <n}(\kappa)^2_\omega$ abbreviate $\kappa\to_{hc}(\kappa)^2_\omega$ via paths of length $<n$. More precisely, it asserts: for any $c: [\kappa]^2\to \omega$, there exists $A\in [\kappa]^\kappa$ and $i\in \omega$ such that for any $C\in [A]^{<\kappa}$ and $\alpha,\beta\in A \backslash C$, there exist $l<n$ and a path $\langle \gamma_k: k<l+1\rangle \subseteq A\backslash C$ with $\gamma_0=\alpha$ and $ \gamma_{l}=\beta$ such that for all $j<l$, $c(\gamma_{j}, \gamma_{j+1})=i$.
\end{definition}

The organization of the paper is 

	\begin{enumerate}
	\item In Section \ref{section: nonch}, we establish the consistency of $\aleph_2\to_{hc} (\aleph_1)^2_{\omega} $ and $\neg \mathsf{CH}$,
	\item in Section \ref{section: 2-precip}, we isolate and investigate 2-precipitous ideals (see Definition \ref{definition: 2-precip}) whose existence implies $\kappa\to_{hc}[\kappa]^2_{\omega,2}$,
	\item in Section \ref{section: consistency2-precip}, we demonstrate two methods of constructing 2-precipitous ideals and show that $\aleph_2\to_{hc}[\aleph_2]^2_{\omega,2} + 2^{\aleph_0} \geq \aleph_2$ is consistent,
	\item in Section \ref{section: closedideal}, we deduce the consistency of $\aleph_2\to_{hc}(\aleph_2)^2_{\omega}$ from an ideal hypothesis,
	\item in Section \ref{section: consistencyClosed}, we sketch how to use large cardinals to establish the consistency of the ideal hypothesis used in Section \ref{section: closedideal},
	\item in Section \ref{section: lengthPaths}, we show we cannot improve the result in Section \ref{section: consistencyClosed} by making the lengths of the paths required to connect vertices shorter,
	\item finally in Section \ref{section: questions}, we finish with some open questions.
	\end{enumerate}

\medskip

\section{The consistency of $\aleph_2\to_{hc} (\aleph_1)^2_{\omega} + \neg \mathsf{CH}$}\label{section: nonch}
We call an ideal $I$ on $\omega_1$ \emph{$\aleph_1$-proper with respect to $S\subseteq P_{\aleph_2}(H(\theta))$} where $\theta$ is a large enough regular cardinal if for any $M\in S$ and $X\in M\cap I^+$, there exists an extension $Y\subseteq_I X$ such that $Y$ is $(M,I^+)$-generic, meaning that for any dense $D\subseteq I^+$ with $D\in M$ and any $Y'\subseteq_I Y$, there exists $Z\in D\cap M$ such that $Y'\cap Z\in I^+$.

\begin{lemma}\label{lemma: properness}
If $I$ is $\aleph_1$-proper with respect to $\{M\}$ with $M\prec H(\theta)$ of size $\aleph_1$ containing $I$  then whenever $Y\in I^+$ is a $(M, I^+)$-generic condition, the following holds: 
\newline
for any $E\subseteq I^+$ in $M$, if there is some $Y'\in E$ such that $Y\subseteq_I Y'$, then there exists $Z\in E\cap M$ such that $Y\cap Z\in I^+$.
\end{lemma}

\begin{proof}
Define a dense subseteq $D_E\subseteq I^+$ in $M$ as follows: $A\in D_E$ iff either there exists $B\in E$, $A\subseteq_I B$ or for all $B\in E$, $A\cap B =_I \emptyset$. By the hypothesis, there exists $Z'\in D_E\cap M$ such that $Y\cap Z'\in I^+$. Note that there exists $Z\in E$ such that $Z'\subseteq_I Z$ since $Z'\cap Y \neq_I \emptyset$ and $Y\subseteq_I Y'\in E$. The elementarity of $M$ then guarantees the existence of such $Z\in M$.
\end{proof}

If $I$ is $\aleph_2$-saturated then $I$ is $\aleph_1$-proper with respect to a closed unbounded subset of $P_{\aleph_2}(H(\theta))$ for sufficiently large $\theta$. There are many models where $\omega_1$ carries a $\sigma$-complete $\aleph_2$-saturated ideal and $\mathsf{CH}$ fails. For example, they are both consequences of Martin's Maximum \cite{MM}.

\begin{proposition}
If there exists a $\sigma$-complete $\aleph_1$-proper ideal on $\omega_1$ with respect to a stationary subset of $\{X\in P_{\aleph_2}(H(\theta)): \sup X\cap \omega_2\in \mathrm{cof}(\omega_1) \}$ for some large enough $\theta$, then $\aleph_2\to_{hc} (\aleph_1)^2_{\omega}$.\footnote{Originally we used the hypothesis that $\omega_1$ carries a countably complete $\aleph_2$-Knaster ideal. Stevo Todorcevic pointed out that our proof should work from a weaker saturation hypothesis, such as $\aleph_2$-saturation.}
\end{proposition}

\begin{proof}
Given $c: [\omega_2]^2\to \omega$, we will find $A\in [\omega_1]^{\aleph_1}$ and $B\in [\omega_2 \setminus \omega_1]^{\aleph_1}$ satisfying the following properties: there exists $k\in \omega$ such that
\begin{enumerate}
\item for any $\alpha_0,\alpha_1\in A$, there are uncountably many $\beta\in B$ such that $c(\alpha_0,\beta)=k=c(\alpha_1, \beta)$, and
\item for any $\beta_0\in B$, there are uncountably many $\alpha\in A$ satisfies that $c(\alpha,\beta_0)=k$.
\end{enumerate}
\begin{claim}
$A\cup B$ is highly connected in the color $k$.
\end{claim}

\begin{proof}[Proof of the claim]
Let $C$ be the countable set of vertices being removed. 
If $\alpha_0, \alpha_1\in A\backslash C$, then it follows immediately by the first requirement that there is some $\beta\in B\backslash C$ such that $c(\alpha_0,\beta)=c(\alpha_1,\beta)=k$.

Let us check the case when $\alpha\in A\backslash C$ and $\beta\in B\backslash C$. We can find some large enough $\alpha'\in A\backslash C$ such that $c(\alpha',\beta)=k$. After that we find some $\beta'\in B \backslash C$ such that $c(\alpha',\beta')=k=c(\alpha,\beta')$. Then $\alpha$ is connected to $\beta$ via the $k$-path: $\alpha, \beta', \alpha',\beta$.

If $\beta_0,\beta_1\in B\backslash C$, then we can easily reduce to the previous case by finding some large enough $\alpha_0\in A\backslash C$ such that $c(\alpha_0, \beta_0)=k$. Apply the previous analysis to $\alpha_0\in A\backslash C, \beta_1\in B\backslash C$.
\end{proof}

We proceed to find $A, B, k$ as above. 
For each $\alpha\in \omega_2\backslash\omega_1$ and $i\in \omega$, let 
$$X_{\alpha,i}=\{\gamma\in \omega_1: c(\gamma,\alpha)=i\}.$$
Let $M\prec H(\theta)$ of size $\aleph_1$ contain $I, c$ with $\sup M\cap \omega_2\in \mathrm{cof}(\omega_1)$ and let $Y\in I^+$ be $(M,I^+)$-generic. Let $\rho \in \omega_2 \backslash \sup M\cap \omega_2$. By the $\sigma$-completeness of $I$, find some $k\in \omega$ such that $A=Y\cap X_{\rho, k}\in I^+$, which is still $(M,I^+)$-generic, since it extends $Y$ which is $(M, I^+)$-generic. Finally, let us define $B$ recursively. Let $\langle a^i=(a^i_0, a^i_1): i<\omega_1\rangle$ enumerate $[A]^2$ with unbounded repetitions.
Suppose we have defined $\langle \beta_j: j<\alpha\rangle$ for some $\alpha<\omega_1$ satisfying that for all $j<\alpha$,
\begin{enumerate}
\item $a^j\subseteq X_{\beta_j, k}$ and 
\item $X_{\beta_j,k}\cap A\in I^+$.
\end{enumerate}

It is clear if we may extend the construction through all $\alpha<\omega_1$, then $A, B=\{\beta_j: j<\omega_1\}$ and $k$ are as desired.

Suppose we are at the $\alpha$-th step of the construction and let us find $\beta_\alpha$ maintaining the same requirements. Let $\bar{\beta}=\min (M \backslash\sup_{j<\alpha}\beta_j)<\sup M\cap \omega_2$. Consider 
$E=\{X_{\beta,k}\in I^+: a^\alpha_0, a^\alpha_1\in X_{\beta,k}, \beta>\bar{\beta}\}$. In particular, $E\in M$ and $A\subseteq_I X_{\rho,k}\in E$. By Lemma \ref{lemma: properness}, there exists a $\beta>\bar{\beta}$ such that $X_{\beta,k}\in M\cap E$ such that $A\cap X_{\beta,k}\in I^+$; letting $\beta_\alpha=\beta$ completes the $\alpha$-th step.
\end{proof}

\section{2-precipitous ideals on $\kappa$ and $\kappa\to_{hc} [\kappa]^2_{\omega,2}$}\label{section: 2-precip}

Chris Lambie-Hanson \cite{lambiehanson2022note} showed that adding weakly compact many Cohen reals forces that $2^\omega\to_{hc} [2^\omega]^2_{\omega,2}$, in contrast with the {\sf ZFC} fact that $2^\omega \not\to_{hc} (2^\omega)^2_{\omega}$. He also demonstrated that such partition relations already have non-trivial consistency strength, by showing that $\square(\kappa)$ implies $\kappa\not\to_{hc}[\kappa]^2_{\omega, <\omega}$.

In this section, we investigate the ideal hypothesis on $\kappa$ that implies $\kappa\to_{hc} [\kappa]^2_{\omega,2}$. In particular, such analysis enables us to have more consistent scenarios, such as a model where $2^{\aleph_0}\geq \aleph_2$ and $\aleph_2\to_{hc} [\aleph_2]^2_{\omega,2}$ both hold.

\begin{definition}\label{definition: 2-precip}
We say an ideal $I$ on $\kappa$ is \emph{2-precipitous} if Player Empty does not have a winning strategy in the following game $G_{I}$ with perfect information: Player Empty and Nonempty take turns playing a $\subseteq$-decreasing sequence of pairs of $I$-positive sets $\langle (A_n, B_n): n\in \omega\rangle$ with Player Empty starting the game. Player Nonempty wins iff there exist $\alpha<\beta$ with $\alpha\in \bigcap_{n\in \omega} A_n$ and $\beta\in \bigcap_{n\in \omega} B_n$.
\end{definition}

\begin{lemma}\label{lemma: denseequivalence}
Fix a dense subset $D\subseteq P(\kappa)/I$.
Player Empty has a winning strategy in $G_I$ iff Player Empty has a winning strategy $\sigma$ in $G_I$ such that $range(\sigma)\subseteq \{A \backslash M: A\in D, M\in I\}$.
\end{lemma}

\begin{proof}
Let us show the nontrivial direction $(\rightarrow)$. Fix a winning strategy $\sigma$ of Player Empty. The input of $\sigma$ will be $(I^+\times I^+)^{<\omega}$, corresponding to the sequence of positive sets Player Nonempty has played so far.
Let $\pi: I^+ \to I^+$ be a map such that $\pi(B)=A \backslash M$ where $(A,M)\in D\times I$ is least (with respect to some fixed well ordering) such that $A\backslash M\subseteq B$. Such $\pi$ exists since $D$ is dense in $P(\kappa)/I$. Consider $\sigma'=\pi\circ \sigma$. Clearly, the range of $\sigma'$ is a subset of $\{A\backslash M: A\in D, M\in I\}$. To see that it is a winning strategy for Player Empty, suppose $\langle (A_n, B_n): n\in \omega\rangle$ is a play such that Player Empty plays according to $\sigma'$. Notice that $\langle (A'_n, B'_n): n\in \omega\rangle$, where $(A'_n, B'_n)=(A_n, B_n)$ when $n$ is odd and $(A'_n,B'_n)=\sigma(\langle (A_{2k-1}, B_{2k-1}): 2k-1<n\rangle)$ is a legal play where Player Empty is playing according to $\sigma$. As a result, there do not exist $\alpha<\beta$ such that $\alpha\in \bigcap_{n\in \omega} A_n' \subseteq \bigcap_{n\in \omega} A_n$ and $\beta\in \bigcap_{n\in \omega} B_n'\subseteq\bigcap_{n\in \omega} B_n$. Therefore, $\sigma'$ is a winning strategy for Player Empty.
\end{proof}

\begin{theorem}\label{theorem: 2preci}
If $\kappa$ carries a uniform normal $2$-precipitous ideal, then $\kappa\to_{hc}[\kappa]^2_{\omega, 2}$.
\end{theorem}

\begin{proof}
Fix a uniform normal $2$-precipitous ideal $I$ on $\kappa$ and a coloring $c: [\kappa]^2\to \omega$. Given $A,B$ two sets of ordinals, we let $A\otimes B=\{(\alpha,\beta)\in A\times B: \alpha<\beta\}$. We say a pair of $I$-positive sets $(B_0,B_1)$ is \emph{$(i,j)$-frequent} if for any $I$-positive sets $B_0'\subseteq B_0$, $B_1'\subset B_1$, there are
	\begin{itemize}
	\item $\alpha<\beta$ with $\alpha\in B_0', \beta\in B_1'$ such that $c(\alpha,\beta)=i$ and 
	\item $\beta'<\alpha'$ with $\beta'\in B_1'$, $\alpha'\in B_0'$ such that $c(\beta',\alpha')=j$.
	\end{itemize}
	\begin{claim}
	There exists a pair of $I$-positive sets $(B_0, B_1)$ and $i,j\in \omega$ such that $(B_0, B_1)$ is $(i,j)$-frequent.
	\end{claim}
	\begin{proof}[Proof of the Claim]
	
	Starting with a positive pair $(A_0, A_1)$, we find some $i\in \omega$ and positive $(C_0,C_1)\subseteq (A_0, A_1)$ such that $(C_0, C_1)$ satisfies the first requirement of the $(i,j)$-frequency, namely, for all positive sets $C_0'\subseteq C_0, C_1'\subseteq C_1$ there are $(\alpha,\beta)\in C_0'\otimes C_1'$ such that $c(\alpha,\beta)=i$. Suppose for the sake of contradiction that such $(C_0,C_1)$ and $i$ do not exist. We define a strategy $\sigma$ for Player Empty:  they start by playing $(A^0, B^0)=_{def}(A_0,A_1)$. At stage $2i$, denoting the game played so far is $\langle (A^k, B^k): k<2i\rangle$, by the hypothesis, there are positive $(A',B')\subseteq(A^{2i-1}, B^{2i-1})$ such that no $(\alpha,\beta) \in A'\otimes B'$ satisfies $c(\alpha,\beta)=i$. Player Empty then plays $(A^{2i}, B^{2i})=(A',B')$. Since by the hypothesis of $I$, Player Empty does not have a winning strategy, there is a play $\langle (A^n,B^n): n\in \omega\rangle$ where Player Empty plays according to the strategy $\sigma$ but in the end, there are $(\alpha,\beta) \in \bigcap_{n\in \omega} A^n \otimes \bigcap_{n\in \omega} B^n$. However, if $c(\alpha,\beta)=k$, then at stage $2k$, the strategy of Empty makes sure $c'' A^{2k}\otimes B^{2k}$ omits $\{k\}$, which is a contradiction.
	
Finally, we repeat the previous argument with input $(C_1, C_0)$ in place of $(A_0, A_1)$ to find positive $(B_1, B_0)\subseteq (C_1, C_0)$ and $j\in \omega$ satisfying the second condition of the $(i,j)$-frequency, as desired.
	\end{proof}		
	Fix an $(i,j)$-frequent pair $(B_0,B_1)$. We strengthen this property of $(B_0, B_1)$ by using the normality of $I$. Recall that for any positive $S\in I^+$, $I^*\restriction S$ denotes the dual filter of $I$ restricted to $S$.
	\begin{claim}\label{claim: normality}
	For any $I$-positive $B_0'\subseteq B_0, B_1'\subseteq B_1$, 
	\begin{itemize}
	\item $\{\alpha\in B_0: \{\beta\in B_1': c(\alpha,\beta)=i\}\in I^+\}\in I^*\restriction B_0$,
	\item $\{\beta'\in B_1: \{\alpha'\in B_0': c(\beta',\alpha')=j\}\in I^+\}\in I^*\restriction B_1$.
	\end{itemize}
	\end{claim}
	\begin{proof}[Proof of the claim]
	Let us just show the first part; the proof of the second part is identical. Suppose for the sake of contradiction that $B^0=_{def}\{\alpha\in B_0: B^1_{\alpha}=_{def}\{\beta\in B_1': c(\alpha,\beta)=i\}\in I\}\in I^+$. Since $I$ is normal, $B^1=\bigtriangledown_{\alpha\in B^0} B^1_\alpha\in I$. Applying the assumption that $(B_0, B_1)$ is $(i,j)$-frequent to $B^0$ and $B_1'\backslash B^1$, we get $(\alpha,\beta)\in B^0 \otimes (B_1' \backslash B^1)$ such that $c(\alpha,\beta)=i$. However, by the definition of $B^1$, $\beta\in B^1$, which is a contradiction.
	\end{proof}
	
Applying Claim \ref{claim: normality}, we find $B_0^*\in I^*\restriction B_0, B_1^*\in I^*\restriction B_1$ such that 
\begin{enumerate}
\item for any $\alpha\in B^*_0$, $\{\beta\in B^*_1: c(\alpha,\beta)=i\}\in I^+$ and 
\item for any $\beta'\in B^*_1$, $\{\alpha'\in B^*_0: c(\beta',\alpha')=j\}\in I^+$.
\end{enumerate}
Let us check that $(B_0^* \cup B^*_1, c^{-1}(\{i,j\}))$ is a highly connected subgraph of size $\aleph_2$. Given $C\in [B_0^* \cup B^*_1]^{\leq \aleph_1}$, $\alpha,\beta\in B_0^* \cup B^*_1 \backslash C$, we need to find an $(i,j)$-valued path connecting them in $B_0^* \cup B^*_1 \backslash C$. Consider the following cases.
\begin{itemize}
\item $\alpha\in B^*_0, \beta\in B^*_1$: let $A_\alpha =\{\gamma\in B^*_1: c(\alpha,\gamma)=i\}\in I^+$ and let $B_\beta=\{\eta\in B^*_0: c(\beta, \eta)=j\}\in I^+$. Since $(B^*_0, B^*_1) $ is $(i,j)$-frequent, we can find $(\gamma,\eta)\in (A_\alpha \backslash C)\otimes (B_\beta \backslash C)$ such that $c(\gamma,\eta)=j$. Then the path $\alpha, \gamma, \eta, \beta$ is as desired.
\item $\alpha, \beta\in B^*_0$ or $\alpha, \beta\in B^*_1$: we can reduce to the previous case by moving either $\alpha$ or $\beta$ to the other side using an edge of $c$-color either $i$ or $j$.
\end{itemize}
\end{proof}

\section{The consistency of the existence of a 2-precipitous ideal}\label{section: consistency2-precip}

In this section we discuss two forcing constructions of a 2-precipitous ideal on $\kappa$. The first is cardinal preserving and the second involves collapsing cardinals.
First let us record some characterizations of 2-precipitous ideals analogous to those of precipitous ideals in \cite{JechPrikry}.

\begin{definition}
\emph{A tree $T$ of maximal antichains} of $P(\kappa)/I \times P(\kappa)/I$ is a sequence of maximal antichains $\langle \mathcal{A}_n: n\in \omega\rangle$ of $P(\kappa)/I \times P(\kappa)/I$ such that $\mathcal{A}_{n+1}$ refines $\mathcal{A}_n$ for each $n\in \omega$. A branch through $T$ is a decreasing sequence of conditions $\langle b_n: n\in\omega\rangle$ such that $b_n\in \mathcal{A}_n$.
\end{definition}

The proof by Jech and Prikry \cite{JechPrikry} (see also \cite[Proposition 2.7]{ForemanHandbook}) essentially gives the following.

\begin{theorem}[\cite{JechPrikry}]\label{theorem: JechPrikry}
$I$ is 2-precipitous if for any pair of positive sets $(C_0, C_1)$ and a tree $T$ of maximal antichains $\langle \mathcal{A}_n : n\in \omega\rangle$ below $(C_0,C_1)$, there exists a sequence $\langle (A_n,B_n): n\in \omega\rangle$ such that 
	\begin{enumerate}
	\item $\langle (A_n, B_n): n\in \omega\rangle$ is a branch through the tree $T$, and 
	\item there exist $\alpha<\beta$ such that $\alpha\in \bigcap_{n\in \omega} A_n$ and $\beta\in \bigcap_{n\in \omega} B_n$.
	\end{enumerate}
\end{theorem}

\begin{remark}\label{remark: wlogdense}
Suppose we are given a dense subset $D\subseteq P(\kappa)/I$. By Lemma \ref{lemma: denseequivalence}, it is no loss of generality to assume $ \{(C_0,C_1)\}  \cup \bigcup_{n\in \omega}\mathcal{A}_n\subseteq \{A\backslash M: A\in D, M\in I\}\times \{A\backslash M: A\in D, M\in I\}$.
\end{remark}

For the rest of this section, we will apply Remark \ref{remark: wlogdense} liberally. Also it turns out that suppressing the quotiented ideal $I$ does not affect the reasoning. Therefore, to avoid cumbersome notations, we will further assume that $\mathcal{A}_n\subseteq D\times D$ for all $n\in \omega$. Given a partial order $\mathbb{P}$, we denote the complete Boolean algebra generated by $\mathbb{P}$ as $\mathbb{B}(\mathbb{P})$.

\begin{proposition}
If $I$ is a $\kappa$-complete normal ideal on $\kappa$ and $P(\kappa)/I\simeq \mathbb{B}(Add(\omega,\lambda))$ for some $\lambda$, then $I$ is 2-precipitous.
\end{proposition}

\begin{proof}
Let $\pi: \mathbb{B}(Add(\omega,\lambda))\to P(\kappa)/I$ be an isomorphism. For each $r\in Add(\omega,\lambda)$, let $X_r = \pi(r)$. Here we identify $Add(\omega,\lambda)$ as a dense subset of $\mathbb{B}(Add(\omega,\lambda))$, $D=\{X_r: r\in Add(\omega,\lambda)\}$.

Suppose for the sake of contradiction that $I$ is not 2-precipitous. By Theorem \ref{theorem: JechPrikry}, there exist a $(C_0,C_1)$ and a tree $T$ of maximal antichains below $(C_0,C_1)$ for which conditions (1) and (2) simultaneously hold. Note that since $P(\kappa)/I\times P(\kappa)/I$ is c.c.c, each $\mathcal{A}_n \subseteq D\times D$ is countable. Find $r_0,r_1\in Add(\omega,\lambda)$ such that $C_i=X_{r_i}$ for $i<2$.

 Let $G\subseteq P(\kappa)/I$ be a generic filter containing $C_0$. Then in $V[G]$, there is a generic elementary embedding $j: V\to M$ which can be taken to be the ultrapower embedding with respect to the added generic $V$-ultrafilter extending the dual filter of $I$.

Consider $T'_n=\{B^*: \exists (A^*,B^*)\in j(\mathcal{A}_n), \kappa\in A^*\}$. Note that $\langle T'_n: n\in \omega\rangle\in M$ since $V[G]\models {}^{\omega}M\subseteq M$ by \cite[Proposition 2.14]{ForemanHandbook}. Note that $T'_n\subseteq j''V$. This follows from the fact that each $\mathcal{A}_n$ is countable, hence $j(\mathcal{A}_n)=j'' \mathcal{A}_n$. 
\begin{claim}\label{claim: maximal}
In $M$, $T'_n$ is a maximal antichain below $j(C_1)$ for the poset $j(P(\kappa)/I)$. 
\end{claim}
\begin{proof}[Proof of the claim]
Suppose not. By the product lemma, $\mathcal{B}=\{B: \exists (A,B)\in \mathcal{A}_n, A\in G\}$ is a maximal antichain for $$(P(\kappa)/I)^V\simeq \mathbb{B}(Add(\omega,\lambda))$$ below $X_{r_1}$ in $V[G]$. We can enumerate $$\mathcal{B}=\langle X_{p_n}: n\in \omega, r_n\in Add(\omega,\lambda)\rangle.$$ In particular, $\langle p_n: n\in \omega\rangle$ is a maximal antichain for $Add(\omega,\lambda)$ below $r_1$ in $V[G]$.

 If $\langle j(X_{p_n}): n\in\omega\rangle$ is not a maximal antichain in $j(P(\kappa)/I)\simeq j(\mathbb{B}(Add(\omega,\lambda)))$, then there exists a condition $r\in Add(\omega, j(\lambda))$ such that $X_r^*=_{\mathrm{def}}j(\pi)(r)$ is incompatible with any condition in the set $\{j(X_{p_n}): n\in\omega\}$. This means $r$ is incompatible with any condition in $\{j(p_n): n\in \omega\}$. Since $j(p_n)=j''p_n$, we may assume $r\in Add(\omega,j''\lambda)$. Let $r^*=j^{-1}(r)$. Then $r^*\in Add(\omega,\lambda)/r_1$ is incompatible with any condition in $\{p_n: n\in \omega\}$. This contradicts with the fact that $\langle p_n: n\in \omega\rangle$ is a maximal antichain subset of $Add(\omega,\lambda)$ below $r_1$ in $V[G]$.
\end{proof}
Let $H\subseteq j(P(\kappa)/I)$ be a generic filter over $V[G]$ containing $j(D)$. Since $M\models j(P(\kappa)/I)$ is a $j(\kappa)$-complete and $\aleph_1$-saturated ideal, $H$ gives rise to an ultrapower embedding $k: M\to N$ with critical point $j(\kappa)$. Consider $b=\{(A_n,B_n): (\kappa,j(\kappa))\in k(j(A_n))\times k(j(B_n))\}$. By Claim \ref{claim: maximal}, $H$ meets $T'_n$ for all $n\in \omega$. As a result, $k\circ j(b)=k\circ j''b$ is a branch $\langle (A^*_n, B^*_n): n\in \omega\rangle$ through $k(j(T))$ in $V[G*H]$ with $(\kappa,j(\kappa))\in \bigcap_{n\in \omega} A^*_n \otimes \bigcup_{n\in \omega} B^*_n$. Since $N$ is well-founded, there is such a branch in $N$. By the elementarity of $k\circ j$, $T$ has a branch $\langle (A_n, B_n): n\in \omega\rangle$ in $V$ for which there are $\alpha<\beta$ with $\alpha\in \bigcap_{n\in \omega} A_n$ and $\beta\in \bigcap_{n\in \omega} B_n$.
\end{proof}

It is easy to see that if $P(\kappa)/I$ has a $\sigma$-closed dense subset, then $I$ is 2-precipitous. However, in this case, it is necessary that $2^{\aleph_0} < \kappa$.

The second construction gives a scenario where $\kappa$ is a small uncountable cardinal (like $\aleph_2$) and $\mathrm{CH}$ fails. In particular, such an ideal can be constructed using the Mitchell collapse \cite{Mitchell}.\footnote{We thank Spencer Unger for his suggestion on the relevance of the Mitchell collapse.}

Recall the Mitchell forcing from \cite{Mitchell} (the representation of the forcing here is due to Abraham, see \cite{Abraham} and \cite[Section 23]{CummingsSurvey}) $\mathbb{M}(\omega,\lambda)$ consists of conditions of the form $(p,r)$ where $p\in Add(\omega,\lambda)$ and $r$ is a function on $\lambda$ of countable support such that for any $\alpha<\lambda$, $\Vdash_{Add(\omega,\alpha)} r(\alpha)$ is a condition in $Add(\omega_1,1)$. The order is that $(p_2, r_2)\leq (p_1, r_1)$ iff $p_2\supset p_1$, $supp(r_2)\supset supp(r_1)$, and for any $\alpha\in supp(r_1)$, $p_2\restriction \alpha\Vdash_{Add(\omega,\alpha)} r_2(\alpha)\leq_{Add(\omega_1,1)} r_1(\alpha)$. 

Define $R$ to be the poset consisting of countably supported functions $r$ with domain $\lambda$ such that for each $\alpha\in supp(r)$, $r(\alpha)$ is an $Add(\omega, \alpha)$-name for a condition in $Add(\omega_1, 1)$. The order of $R$ is the following: $r_2\leq_R r_1$ iff $supp(r_2)\supset supp(r_1)$ and for any $\alpha\in supp(r_1)$, $\Vdash_{Add(\omega,\alpha)} r_2(\alpha)\leq_{Add(\omega_1, 1)} r_1(\alpha)$.

The following are standard facts about this forcing (see \cite{CummingsSurvey}): 
\begin{enumerate}
\item $\mathbb{M}(\omega,\lambda)$ projects onto $Add(\omega,\lambda)$ by projecting onto the first coordinate,
\item $Add(\omega,\lambda)\times R$ projects onto $\mathbb{M}(\omega,\lambda)$ by the identity map.
\end{enumerate}

\begin{remark}
Whenever $(p_2, r_2)\leq_{\mathbb{M}(\omega,\lambda)}(p_1, r_1)$, there exists $r_2'\in R$ with $dom(r_2')=dom(r_2)$ such that $r_2'\leq_{R} r_1$ and $p_2\restriction \alpha\Vdash_{Add(\omega,\alpha)} r_2(\alpha)=r_2'(\alpha)$ for any $\alpha\in dom(r_2)$. In other words, $(p_2, r_2)$ and $(p_2, r_2')$ are equivalent conditions. We will use this fact freely in the following proofs.
\end{remark}

Note that the poset $R$ has the property that any countable decreasing sequence has a greatest lower bound.

\begin{proposition}\label{proposition: mitchell}
If $P(\kappa)/I\simeq \mathbb{B}(\mathbb{M}(\omega,\lambda))$ for some $\lambda$, then $I$ is 2-precipitous. 
\end{proposition}

\begin{proof}
Let $\pi: \mathbb{M}(\omega,\lambda)\to D$ be an isomorphism where $D\subseteq P(\kappa)/I$ is a dense subset. For any $(p,r)\in \mathbb{M}(\omega,\lambda)$, let $X_{p,r}$ denote $\pi(p,r)$. Assume for the sake of contradiction that $I$ is not 2-precipitous. Fix a winning strategy $\sigma$ for Player Empty in the game $G_I$. We may assume $\sigma$ satisfies the conclusion of Lemma \ref{lemma: denseequivalence} applied to $D$. To avoid cumbersome notations, we will assume for simplicity that $\sigma$ outputs elements from $D\times D$, as discussed after Remark \ref{remark: wlogdense}.
 We will use $\sigma$ to construct a tree of antichains $T=\langle \mathcal{A}_n: n\in\omega\rangle$ below $\sigma(\emptyset)=(A,B)=(X_{p^{-1}_a, r^{-1}_a}, X_{p^{-1}_b, r^{-1}_b})=\mathcal{A}_{-1}$ satisfying the following additional properties: 
	\begin{enumerate}
	\item $\mathcal{A}_{n+1}$ refines $\mathcal{A}_n$,
	\item $\mathcal{A}_n=\langle (a^n_i, b^n_i):i<\gamma_n\rangle\subseteq D$ is countable (note that it is \emph{not maximal} anymore),
	\item for each $n\in \omega,i\in \gamma_n$, there are unique $(p^n_{i,a}, r^n_{i,a}),(p^n_{i,b}, r^n_{i,b}) \in \mathbb{M}(\omega,\lambda)$ such that $X_{p^n_{i,a}, r^n_{i,a}}=a^n_i$ and $X_{p^n_{i,b}, r^n_{i,b}}=b^n_i$,
	\item for any $n$ and $i<j\in \gamma_n$, $(r^n_{j,a}, r^n_{j,b})\leq_{R\times R} (r^n_{i,a}, r^n_{i,b})$,
	\item for any lower bound $(r_0, r_1)$ for $\langle (r^n_{i,a}, r^n_{i,b}): i\in \gamma_n\rangle$, we have that $\mathcal{A}_n \downarrow (r_0, r_1)=_{def} \{X_{p^n_{i,a}, r_0}, X_{p^n_{i,b}, r_1}: i\in \gamma_n\}$ is a maximal antichain in $(A\cap X_{\emptyset, r_0}, B\cap X_{\emptyset,r_1})$,
	\item for any $n$ and $i,j$, $(r^{n+1}_{j,a}, r^{n+1}_{j,b})\leq_{R\times R} (r^n_{i,a}, r^n_{i,b})$,
	\item for any branch $\langle (A_n,B_n): n\in \omega\rangle$ through $T$, there do not exist $\alpha<\beta$ such that $\alpha\in \bigcap_{n\in\omega} A_n$ and $\beta\in \bigcap_{n\in \omega} B_n$.
 	\end{enumerate}	
 Assuming that the construction of such $T$ is possible, let us derive a contradiction. Let $(r_a, r_b)$  be the greatest lower bound in $R\times R$ for $\langle \langle (r^n_{i,a}, r^n_{i,b}): i\in \gamma_n\rangle: n\in \omega\rangle$. By property (5), we know that for each $n$, $\mathcal{A}_n\downarrow (r_a, r_b)$ is a maximal antichain below $(X_{\emptyset, r_a}\cap A, X_{\emptyset, r_b}\cap B)$ as a subset of $(P(\kappa)/I)^V$ in $V[G]$.

Force over $V$ to get a generic $G\subseteq P(\kappa)/I$ over $V$ containing $X_{\emptyset,r_a}\cap A$.
Using $G$, we find an elementary embedding $j: V\to M$ with critical point $\kappa$. In $V[G]$, consider the tree $T'$ consisting of $\mathcal{A}_n'=\{j(C): \kappa\in j(D), (D,C)\in \mathcal{A}_n \downarrow (r_a, r_b)\}$. Notice that by property (5) and the product lemma, $\mathcal{A}^*_n=\{C: j(C)\in \mathcal{A}'_n\}$ is a maximal antichain below $X_{\emptyset, r_b}\cap B$.
	\begin{claim}\label{claim: maximal}
	For each $n\in \omega$, $\mathcal{A}_n' \subseteq j(P(\kappa)/I)$ is a maximal antichain below $j(B)\cap X_{\emptyset,j(r_b)}$ in $V[G]$ (and in $M$, since $V[G]\models {}^\omega M\subseteq M$).
	\end{claim} 
	\begin{proof}[Proof of the claim]
	Otherwise, we can find $(p, r^*)\in j(\mathbb{M}(\omega,\lambda))$ below $(\emptyset, j(r_b))$ such that $X_{p,r^*}^*=_{\mathrm{def}}j(\pi)((p,r^*))\subseteq j(B)\cap X_{\emptyset, j(r_b)}$ and $X_{p,r^*}^*$ is incompatible with any element in $\mathcal{A}'_n$. By changing to an equivalent condition if necessary, we may assume that $r^*\leq_{j(R)} j(r_b) $. As a result, $p\perp_{j(Add(\omega,\lambda))}j(p^n_{k,b})$ for all $k\in \omega$. Consider $p'=j^{-1}(p)\in Add(\omega,\lambda)$. Then $p'\perp_{Add(\omega,\lambda)} p^n_{k,b}$ for all $k\in \omega$. As a result, $X_{p', r_{b}}$ is incompatible with each element in $\mathcal{A}^*_n$, but $X_{p', r_b}\cap B\cap X_{\emptyset, r_b}\in I^+$, which is a contradiction to the fact that $\mathcal{A}^*_n$ is a maximal antichain below $X_{\emptyset, r_b}\cap B$.
	\end{proof}

 Let $H\subseteq j(P(\kappa)/I)$ be a $V$-generic filter containing $j(B\cap X_{\emptyset, r_b})$. Then in $V[G*H]$, we can form an elementary embedding $k: M\to N$ with critical point $j(\kappa)$. Consider $b=\{(A_n,B_n): (\kappa,j(\kappa))\in j(A_n)\otimes k(j(B_n)), (A_n, B_n)\in \mathcal{A}_n, n\in \omega\}$. By Claim \ref{claim: maximal}, $k\circ j''b\in V[G*H]$ is a branch through $k(j(T))$ violating property (7) as witnessed by $(\kappa, j(\kappa))$. Since $N$ is a well-founded inner model of $V[G*H]$, there is such a branch in $N$. By the elementarity of $k\circ j$, there is such a branch in $V$ through $T$ violating property (7), which is a contradiction.

Let us turn to the construction of $T$. We will construct $T$ levelwise recursion. Let $\mathcal{A}_{-1}=\sigma(\emptyset)=(A,B)=(X_{p^{-1}_a, r^{-1}_a}, X_{p^{-1}_b, r^{-1}_b})$. To avoid excessive repetitions, we will assume that all the conditions from $\mathbb{M}(\omega,\lambda)\times \mathbb{M}(\omega,\lambda)$ extend $((p^{-1}_a, r^{-1}_a),(p^{-1}_b, r^{-1}_b))$.

Let us first define $T(0)=\mathcal{A}_0$. Recursively, suppose we have defined $\mathcal{A}_{0,<\eta}=\langle (X_{p^0_{i,a}, r^0_{i,a}}, X_{p^0_{i,b}, r^0_{i,b}}): i<\eta\rangle$ (partially) satisfying property (4). Let $(t_0,t_1)$ be a lower bound for $\langle (r^0_{i,a}, r^0_{i,b}): i<\eta\rangle$ in $R\times R$. If there exists $(q_0, t'_0)\leq (p_a^{-1}, t_0), (q_1, t'_1)\leq (p_b^{-1}, t_1)$ such that $(X_{q_0, t'_0}, X_{q_1, t'_1})$ is incompatible with any element in $\mathcal{A}_{0,<\eta}$, let $(Y^0_{\eta,a}, Y^0_{\eta,b})$ be one such $(X_{q_0, t'_0}, X_{q_1, t'_1})$.
Then we define $(X_{p^0_{\eta,a}, r^0_{\eta,a}}, X_{p^0_{\eta,b}, r^0_{\eta,b}})$ to be $\sigma(\langle \emptyset, (Y^0_{\eta,a}, Y^0_{\eta,b})\rangle)$. Notice that this process must stop at some countable stage $\gamma_0<\omega_1$ since $\{(p^0_{i,a}, p^0_{i,b}): i<\gamma_0\}$ is an antichain in $Add(\omega, \lambda)\times Add(\omega,\lambda)$ below $(p_a^{-1}, p_b^{-1})$, which satisfies the countable chain condition. Let us verify all the properties are satisfied. Properties (1), (2), (3), (4) and (6) are satisfied by the construction. Property (7) is not relevant at this stage. Property (5) is satisfied since we only stop when the process described above cannot be continued, which is exactly saying property (5) is satisfied.

In general, the definition of $\mathcal{A}_{n+1}$ is very similar to the construction above. Basically, for each $(C_0, C_1)\in \mathcal{A}_n$, we repeat the process above with $(C_0,C_1)$ playing the role of $\mathcal{A}_{-1}$. One difference, in order to satisfy property (6), is that at the beginning of the construction, we let $(h_0, h_1)$ be the lower bound in $R\times R$ for $\langle (r^n_{i,a}, r^n_{i,b}): i<\gamma_n\rangle$ and work below $((p_a^{-1}, h_0), (p_b^{-1}, h_1))$ in $\mathbb{M}(\omega,\lambda)\times \mathbb{M}(\omega,\lambda)$.

Finally, to see that property (7) is satisfied, notice that any branch $b$ through $T$ corresponds to a play of the game $G_I$ where Player Empty is playing according to their winning strategy $\sigma$. More precisely, $b$ is the sequence of sets played by Player Empty according to $\sigma$ in a play of the game $G_I$.
As a result, the winning condition of Player Empty ensures (7) is satisfied. \end{proof}

\begin{corollary}
It is consistent that $\aleph_2\to_{hc}[\aleph_2]^2_{\omega,2}$ and $2^{\aleph_0}\geq \aleph_2$.
\end{corollary}

\begin{proof}
Let $\kappa$ be a measurable cardinal. Then in $V^{\mathbb{M}(\omega,\kappa)}$, $2^{\aleph_0}\geq \aleph_2$ and there is an ideal satisfying the hypothesis of Proposition \ref{proposition: mitchell} (see for example \cite[Theorem 23.2]{CummingsSurvey}). Apply Proposition \ref{proposition: mitchell} and Theorem \ref{theorem: 2preci}.
\end{proof}

\section{$\sigma$-closed ideals and monochromatic highly connected subgraphs}\label{section: closedideal}
In this section, we prove the following theorem.
\begin{theorem}\label{theorem: main}
Suppose a regular cardinal $\kappa$ carries a countably complete uniform ideal $I$ such that there exists a dense $\sigma$-closed collection $H\subseteq I^+$. Then $\kappa\to_{hc}(\kappa)^2_\omega$. Moreover, $\kappa\to_{hc, <4}(\kappa)^2_\omega$ holds.
\end{theorem}
Fix an ideal $I$ as in the hypothesis of Theorem \ref{theorem: main}. It is worth comparing such ideal with those of the previous section: 
	\begin{itemize}
	\item we do not insist that $I$ is normal any more,
	\item we impose the stronger requirement that the ideal has a $\sigma$-closed dense subset; note that any such ideal is 2-precipitous.
	\end{itemize}

Fix a coloring $c: [\kappa]^2\to \omega$. Given $B_0, B_1\in I^+$ and $i\in \omega$, we say \emph{$(B_0,B_1)$ is $i$-frequent} if for any positive $B_0'\subseteq B_0$ and positive $B_1'\subseteq B_1$, it is the case that $\{\alpha\in B_0': \{\beta\in B_1': c(\alpha,\beta)=i\}\in I^+\}\in I^+$.

\begin{remark}\label{remark: large}
Equivalently, $(B_0,B_1)$ is $i$-frequent if for any positive $B_1'\subseteq B_1$, it is the case that $\{\alpha\in B_0: \{\beta\in B_1': c(\alpha,\beta)=i\}\in I^+\}\in I^*\restriction B_0$. See the argument in Claim \ref{claim: normality}.
\end{remark}

\begin{claim}\label{claim: split}
There exists a positive $B\in I^+$ and $i\in \omega$ such that for any positive $B'\subseteq B$, there are positive $B_0, B_1\subseteq B'$ such that $(B_0,B_1)$ is $i$-frequent.
\end{claim}
In the following proof, to avoid repetitions, whenever we mention a positive set, we implicitly assume the positive set is in the $\sigma$-closed dense collection $H$.
\begin{proof}
Suppose otherwise for the sake of contradiction.
For $A\in I^+$ and $j\in \omega$, let $(*)_{A,j}$ abbreviate the assertion: there are positive sets $B_0, B_1\subseteq A$ such that $(B_0, B_1)$ is $j$-frequent. By the hypothesis, we can recursively define $\langle B'_k \in I^+: k\in \omega\rangle$ such that 

	\begin{itemize}
	\item $B'_0 = \kappa$,
	\item for any $k\in \omega$, $B'_{k+1}\subseteq B'_k$ and $\neg (*)_{B'_{k+1},k}$.
	\end{itemize}

Let $B'=\bigcap_{k\in \omega} B'_k$. By the $\sigma$-closure of $I$, we have that $B'\in I^+$. Then it satisfies that for any $i\in \omega$, such that there are no positive $B_0, B_1\subseteq B'$ such that $(B_0, B_1)$ is $i$-frequent.

Recursively construct an $\omega$-sequence of pairs of $I$-positive sets $$\langle (C_k,D_k): k\in \omega\rangle$$ as follows: start with $(B',B')=(C_{-1}, D_{-1})$; since it is not $0$-frequent, there are positive $(C_0, D_0)$ such that for all $\alpha\in C_0$, $\{\beta\in D_0: c(\alpha,\beta)=0\}\in I$. In general, as $(C_i, D_i)$ is not $i+1$-frequent, we can find $(C_{i+1}, D_{i+1})$ such that for all $\alpha\in C_i$, $\{\beta\in D_i: c(\alpha,\beta)=i+1\}\in I$. Let $C^*=\bigcap_{i\in \omega} C_i$ and $D^*=\bigcap_{i\in \omega} D_i$. By the $\sigma$-closure of $I$, both $C^*$ and $D^*$ are in $I^+$. By the $\sigma$-completeness of the ideal, we can find some $i$ and $\alpha\in C^*$ such that $\{\beta\in D^*: c(\alpha,\beta)=i\}\in I^+$. However, this contradicts with the assumption that $\alpha\in C_i$.
\end{proof}

To finish the proof of Theorem \ref{theorem: main}, by repeatedly applying Claim \ref{claim: split}, we can find $\langle B_n\in I^+: n\in \omega\rangle$ and $i\in \omega$ such that for any $n<k$, $(B_n, B_k)$ is $i$-frequent. 

Given $n\in \omega$, let $B_n^*\subseteq B_n$ be the collection of $\alpha\in B_n$ satisfying that for any $k>n$, $\{\beta\in B_k: c(\alpha,\beta)=i\}\in I^+$. Notice that $B^*_n=_I B_n$ by Remark \ref{remark: large} and the fact that $I$ is $\sigma$-complete.

We claim that $B=\bigcup_{n\in \omega} B^*_n$ is highly connected in the color $i$.
Given $\alpha<\beta\in B$ and $C\in [B]^{<\kappa}$, there must be some $n_0, n_1 \in \omega$ such that $\alpha\in B^*_{n_0}$ and $\beta\in B^*_{n_1}$. Find $k>\max \{n_0, n_1\}$. By the hypothesis, $C_0=\{\gamma\in B_k^*: c(\alpha,\gamma)=i\}\in I^+$ and $C_1=\{\gamma\in B_{k+1}^*: c(\beta,\gamma)=i\}\in I^+$. As $(B_k, B_{k+1})$ is $i$-frequent, we can find $\gamma_0\in C_0\setminus C$ and $\gamma_1\in C_1 \setminus C$ such that $c(\gamma_0, \gamma_1)=i$. As a result, $\alpha, \gamma_0, \gamma_1, \beta$ is the required path of color $i$.

\section{Remarks on the consistency of the ideal hypothesis}\label{section: consistencyClosed}
For regular cardinal $\lambda\geq \kappa$, if $\kappa$ be $\lambda$-supercompact, we show that in $V^{Coll(\omega_1, <\kappa)}$, $\lambda$ carries a $\kappa$-complete uniform ideal which admits a dense and $\sigma$-closed collection of positive sets. The construction is due to Galvin, Jech, Magidor \cite{GalvinJechMagidor} and, independently to Laver \cite{Laver}. We supply a proof for the sake of completeness.

Let $U$ be a fine normal ultrafilter on $P_{\kappa}\lambda$. By a theorem of Solovay (see \cite[Theorem 14]{HODDichotomy} for a proof), there exists $B\in U$ such that $a\in B\mapsto \sup a$ is injective. Let $j: V\to M\simeq Ult(V, U)$. Let $\delta=\sup j''\lambda$. Let $G\subseteq Coll(\omega_1, <\kappa)$ be generic over $V$. It is well-known that if $H\subseteq Coll(\omega_1,[\kappa, j(\kappa)))$ is generic over $V[G]$, then we can lift $j$ to $j^+: V[G]\to M[G*H]$ in $V[G*H]$. 

In $V[G]$, consider the ideal 
$$I=\{X\subseteq \lambda: \Vdash_{Coll(\omega_1, [\kappa,<j(\kappa)))} \delta \not\in j^+(X)\}.$$ 
The fact that $I$ is $\kappa$-complete and uniform is immediate. Let us show that there exists a dense $\sigma$-closed collection of positive sets.

For each $r\in Coll(\omega_1, <j(\kappa))^M/G$, there exists a function $f_r: B\to P$ such that $j(f_r)(j'' \lambda)=r$. Define $X_r=\{\sup a: a\in B, f_r(a)\in G\}$. 

\begin{claim}
$X_r\in I^+$.
\end{claim}

\begin{proof}
It suffices to check that $r\Vdash \delta\in j^+(X_r)$. Let $H\subseteq Coll(\omega_1, <j(\kappa))^M/G$ containing $r$ be generic over $V[G]$, then we can lift $j$ to $j^+: V[G]\to M[H]$. In particular, $H=j^+(G)$. Since $j(f_r)(j''\lambda)=r\in j^+(G)$, we have that $\delta\in j^+(X_r)$.
\end{proof}

\begin{claim}
$X_r\subseteq_I X_{r'}$ iff $r\leq_{Coll(\omega_1, <j(\kappa))^M/G} r'$.
\end{claim}

\begin{proof}
If $r\leq r'$, then it is clear that $X_r\subseteq_I X_{r'}$. For the other direction, suppose for the sake of contradiction that $r\not\leq_{Coll(\omega_1, <j(\kappa))^M/G} r'$. In particular, there is some extension $r^*$ of $r$ that is incompatible with $r'$. Then $r^*\Vdash \delta\in j^+(X_r)\setminus j^+(X_{r'})$. Hence, $X_{r}\not\subseteq_I X_{r'}$.
\end{proof}

As a result, 
$$\{X_r: r\in Coll(\omega_1, <j(\kappa))^M/G\}$$
is $\sigma$-closed, since $ Coll(\omega_1, <j(\kappa))^M/G$ is $\sigma$-closed in $V[G]$.

\begin{claim}
For any $X\in I^+$, there exists some $r$ such that $X_r\subseteq_I X$.
\end{claim}

\begin{proof}
Let $r\in Coll(\omega_1, <j(\kappa))^M/G$ force that $\delta\in j^+(X)$. We show that $X_r\subseteq_{I} X$. Otherwise, there is some $r'$ forcing that $\delta\in j^+(X_r)\setminus j^+(X)$. In particular, $r'$ forces that $j(f_r)(j''\lambda)=r\in j^+(G)$. By the separability of the forcing, $r'\leq_{Coll(\omega_1, <j(\kappa))^M/G} r$. This contradicts the fact that $r$ forces $\delta\in j^+(X)$.
\end{proof}

The following is now immediate from the preceding arguments, coupled with Theorem \ref{theorem: main}.

\begin{theorem}\begin{enumerate}
\item If $\kappa$ is measurable, then in $V^{Coll(\omega_1, <\kappa)}$, $\aleph_2\to_{hc}(\aleph_2)^2_{\omega}$. 
\item If $\kappa$ is supercompact, then in $V^{Coll(\omega_1, <\kappa)}$, for all regular cardinal $\lambda\geq \aleph_2$, $\lambda\to_{hc}(\lambda)^2_{\omega}$.
\end{enumerate}
\end{theorem}

Some  large cardinal assumption is necessary to establish the consistency of $\kappa\to_{hc}(\kappa)^2_{\omega}$, as shown in \cite{lambiehanson2022note}.

\section{The lengths of the paths}\label{section: lengthPaths}

In Section \ref{section: closedideal}, we have shown that if there exists a $\sigma$-complete uniform ideal on $\omega_2$ admitting a $\sigma$-closed collection of dense positive sets, then $\omega_2\to_{hc, <4} (\omega_2)^2_\omega$. One natural question is  whether we can improve the conclusion to $\omega_2\to_{hc, <3} (\omega_2)^2_\omega$. In this section, we show that the answer is no, at least not from the same hypothesis.

\begin{remark}\label{remark: genericallymeasurable}
If there is a $\sigma$-closed forcing $P$ such that in $V^P$, there is a transitive class $M$ and an elementary embedding $j: V\to M$ with critical point $\kappa$, then $\kappa\to_{hc}(\kappa)^2_\omega$ holds. Essentially the same proof from Section \ref{section: closedideal} works.
\end{remark}

\begin{theorem}
It is consistent relative to the existence of a measurable cardinal that $\aleph_2\to_{hc, <4} (\aleph_2)^2_\omega$ but
 $\aleph_2\not\to_{hc, <3} (\aleph_2)^2_\omega$.
\end{theorem}

\begin{proof}
Let $\kappa$ be a measurable cardinal.
We will make use of a forcing poset $\mathbb{P}_{\kappa}$ due to Komj\'{a}th-Shelah \cite[Theorem 7]{KomjathShelah}. The final forcing will be $Coll(\omega_1, <\kappa)* \mathbb{P}_{\kappa}$. It follows from Komj\'{a}th-Shelah's work that $\aleph_2\not\to_{hc, <3} (\aleph_2)^2_\omega$ in the final model. More specifically, this was proved in Claim 4 inside the proof of Theorem 7 of \cite{KomjathShelah}. Note that $(*)$ in \cite{KomjathShelah} is a consequence of $\aleph_2\to_{hc, <3} (\aleph_2)^2_\omega$.
 
Let $G\subseteq Coll(\omega_1, <\kappa)* \mathbb{P}_{\kappa}$ be a generic filter over $V$. By Remark \ref{remark: genericallymeasurable} applied in $V[G]$ (namely, replacing the occurrences of ``$V$"  there with ``$V[G]$"), it is enough to check that in a further countably closed forcing extension over $V[G]$, there exists an elementary embedding $j: V[G]\to M$ with critical point $\omega_2^{V[G]}=\kappa$.

To that end, let us recall the definition of $\mathbb{P}_\kappa$: conditions consist of $(S, f, \mathcal{H}, h)$ such that 
\begin{itemize}
\item $S\in [\kappa]^{\leq \aleph_0}$,
\item $f: [S]^2\to [\omega]^\omega$,
\item $\mathcal{H}\subseteq [S]^\omega$, $|\mathcal{H}|\leq \aleph_0$, for each $H\in \mathcal{H}$, $otp(H)=\omega$ and for any $H, H'\in \mathcal{H}$, $H\cap H'$ is finite.
\item if $\alpha\in S$, $H\in \mathcal{H}$ with $\min H < \alpha$, then $|\{\beta\in H: h(H)\in f(\alpha,\beta)\}|\leq 1$.
\end{itemize}
The order is: $(S', f', \mathcal{H}', h')\leq (S, f, \mathcal{H}, h)$ iff $S'\supset S$, $\mathcal{H}'\supset \mathcal{H}$, $f'\restriction [S]^2 = f$, $h'\restriction \mathcal{H} = h$ and for any $H\in \mathcal{H}'-\mathcal{H}$, $H\not\subseteq S$. Note that $\mathbb{P}_\kappa$ is a countably closed forcing of size $\kappa$.

Let $j: V\to M$ witness that $\kappa$ is measurable.
Let $G*H\subseteq Coll(\omega_1,<\kappa)*\mathbb{P}_\kappa$. Let $G^*\subseteq Coll(\omega_1, <j(\kappa))/G*H$ be generic over $V[G*H]$. This is possible since $Coll(\omega_1,<\kappa)*\mathbb{P}_\kappa$ regularly embeds into $Coll(\omega_1, <j(\kappa))$ with a countably closed quotient (see \cite[Theorem 14.3]{CummingsSurvey}). As a result, we can lift $j: V[G]\to M[G^*]$. In order to lift further to $V[G*H]$, we need to force $j(\mathbb{P}_\kappa)/H$ over $V[G^*]$. It suffices to show that in $V[G^*]$, $j(\mathbb{P}_\kappa)/H$ is countably closed.

Suppose $\langle p_n=_{def} (S_n, f_n, \mathcal{H}_n, h_n): n\in \omega \rangle \subseteq j(\mathbb{P}_\kappa)/H$ is a decreasing sequence. We want to show that $q=\bigcup_{n\in \omega} p_n$ is the lower bound desired.
 For this, we only need to verify that $q\in j(\mathbb{P}_\kappa)/H$. More explicitly, we need to verify that $q$ is compatible with $h\in H$ for any $h\in H$. For each $p=(S_p,f_p,\mathcal{H}_p, h_p)\in j(\mathbb{P}_\kappa)$, let $p\restriction \kappa$ denote the condition: $(S_p\cap \kappa, f_p\restriction [S_p\cap \kappa]^2, \mathcal{H}_p\cap P(\kappa), h_p\restriction (\mathcal{H}_p\cap P(\kappa)))$. It is not hard to check that $p\restriction \kappa\in \mathbb{P}_\kappa$ and $p\leq_{j(\mathbb{P}_\kappa)} p\restriction \kappa$.

\begin{claim}\label{claim: equivalence}
$p\in j(\mathbb{P}_\kappa)/H$ iff $p\restriction \kappa\in H$.
\end{claim}

\begin{proof}[Proof of the claim]
If $p\restriction \kappa\in H$, to see $p\in j(\mathbb{P}_\kappa)/H$, it suffices to see that for any $r\leq_{\mathbb{P}_\kappa} p\restriction \kappa$ and $r\in H$, $r$ is compatible with $p$. To check that $r\cup p$ can be extended to a condition, it suffices to check that for any $B\in \mathcal{H}_r\setminus\mathcal{H}_p$, $B\not\subseteq S_p$ and any $B\in \mathcal{H}_p\setminus\mathcal{H}_r$, $B\not\subseteq S_r$.
To see the former, note that if $B\in \mathcal{H}_r\setminus \mathcal{H}_p$, then $B\in \mathcal{H}_r \setminus\mathcal{H}_{p\restriction \kappa}$, since $r\leq p\restriction \kappa$, $B\not\subseteq S_p\cap \kappa$. Since $B\subseteq \kappa$, we have $B\not\subseteq S_p$. To see the latter, note that $B\in \mathcal{H}_p\setminus\mathcal{H}_r$ implies that $B\in \mathcal{H}_p \setminus \mathcal{H}_{p\restriction \kappa}$. In particular, $B\cap [\kappa,j(\kappa))\neq \emptyset$. Hence $B\not\subseteq S_r$ since $S_r\subseteq \kappa$.

If $p\in j(\mathbb{P}_\kappa)/H$, then for any $h\in H$, $h$ and $p$ are compatible. In particular, since $p\leq p\restriction \kappa$, we know that any $h\in H$ is compatible with $p\restriction \kappa$. This implies that $p\restriction \kappa\in H$.
\end{proof}

To finish the proof, since for each $n\in \omega$, by Claim \ref{claim: equivalence} and the fact that $p_n\in j(\mathbb{P}_\kappa)/H$, we have that $p_n\restriction \kappa\in H$. As a result, we must have $q\restriction \kappa\in H$. By Claim \ref{claim: equivalence}, we have $q\in j(\mathbb{P}_\kappa)/H$. \end{proof}

\section{Open questions}\label{section: questions}

\begin{question}
Starting from the existence of a weakly compact cardinal, can one force that $\aleph_2\to_{hc}(\aleph_2)^2_{\omega}$?
\end{question}
It is possible to use a weaker assumption than the existence of a measurable cardinal to run the proof of Theorem \ref{theorem: main}. In particular, we can use a strongly Ramsey cardinal instead. Recall that $\kappa$ is \emph{strongly Ramsey} if for any $A\subset \kappa$, there is a $\kappa$-model $M$ containing $A$ (namely, $M$ is a transitive model of $\mathrm{ZFC}^-$ containing $\kappa$ and is closed under $<\kappa$-sequences), such that there exists a $\kappa$-complete $M$-ultrafilter that is weakly amenable to $M$, meaning that for any $\mathcal{F}\in M$ and $|\mathcal{F}|\leq \kappa$, $\mathcal{F}\cap U \in M$. See \cite{Gitman} for more information on Ramsey-like cardinals. However, being a strongly Ramsey cardinal is a much stronger condition than being a weakly compact cardinal.

\begin{question}
Is $\aleph_2\to_{hc,<3} (\aleph_2)^2_{\omega}$ consistent?
\end{question}

\begin{question}
Can one separate $\aleph_2\to_{hc,<m} (\aleph_2)^2_{\omega}$ from $\aleph_2\to_{hc,<n} (\aleph_2)^2_{\omega}$ where $4\leq m<n\leq \omega$?
\end{question}

\begin{question}
Is $\aleph_2\not\to_{hc}(\aleph_1)^2_{\omega}$ consistent?
\end{question}


Our next problem is more open-ended, which concerns the natural generalizations of $2$-precipitous ideals. Recall that we have shown that the existence of a uniform normal $2$-precipitous ideal on $\kappa$ implies that $\kappa\geq \omega_2$.

\begin{problem}
For $n\geq 2$ and an ideal $I$ on $\kappa\geq \aleph_2$, find a natural definition of \emph{$n$-precipitousness} that generalizes Definition \ref{definition: 2-precip} in the case that $n=2$, such that if $\kappa$ carries a uniform normal $n$-precipitous ideal, then $\kappa \geq \aleph_n$.
\end{problem}

The referee suggested the following interesting problem, which is worth further investigation. Fix cardinals $\kappa,\lambda,\theta$ and an ideal $\mathcal{J}$ on $\lambda$ containing $[\lambda]^{<\lambda}$. For any set $X$ of order type $\lambda$, let $\mathcal{J}_X$ be the ideal on $X$ moved from $\mathcal{J}$ using the order isomorphism between $X$ and $\lambda$.

\begin{problem}

Investigate the following partition relations: $\kappa\to_{\mathcal{J}-hc}(\lambda)^2_\theta$ which asserts for any $c: [\kappa]^2\to \theta$, there are $X\subset \kappa$ of order type $\lambda$ and $i\in \theta$ such that for any $J\in \mathcal{J}_X$, $(X\setminus J, c^{-1}(i)\cap [X\setminus J]^2)$ is connected.
\end{problem}

\section{Acknowledgement}
We thank Stevo Todorcevic and Spencer Unger for helpful discussions and comments. We are grateful to the anonymous referee, whose helpful suggestions, corrections and comments greatly improve the exposition of this paper.

\bibliographystyle{amsplain}
\bibliography{bib}

\end{document}